\documentclass[a4paper,12pt,reqno]{article}
\usepackage{amssymb}
\usepackage{amsfonts,amsmath,amstext,amsbsy,amsopn,amsthm,amscd}

\usepackage{graphicx}
\usepackage{color}
\usepackage{dsfont}

\newcommand{\Ni}{\mathds{N}}

\newcommand{\Ci}{\mathds{C}}

\newcounter{teller}

\newenvironment{tabel}{\begin{list}%
{\rm  (\alph{teller})\hfill}{\usecounter{teller} \leftmargin=1.1cm
\labelwidth=1.1cm \labelsep=0cm \parsep=0cm}
                      }{\end{list}}

\newcounter{tellerr}

\newenvironment{tabeleq}{\begin{list}%
{\rm  (\roman{tellerr})\hfill}{\usecounter{tellerr} \leftmargin=1.1cm
\labelwidth=1.1cm \labelsep=0cm \parsep=0cm}
                         }{\end{list}}

\newcounter{tellerrr}

\newenvironment{tabelR}{\begin{list}%
{\rm  (\Roman{tellerrr})\hfill}{\usecounter{tellerrr} \leftmargin=1.1cm
\labelwidth=1.1cm \labelsep=0cm \parsep=0cm}
                         }{\end{list}}

\newtheorem{theorem}{Theorem}[section]
\newtheorem{lemma}[theorem]{Lemma}

\newtheorem{proposition}[theorem]{Proposition}

\theoremstyle{definition}

\newtheorem{definition}[theorem]{Definition}

\newtheorem{exam}[theorem]{Example}

\newtheorem{rem}[theorem]{Remark}

\theoremstyle{remark}

\numberwithin{equation}{section}

\newcommand{\field}[1]{\mathbb{#1}}

\newcommand{\R}{\field{R}}

\DeclareMathOperator{\dom}{dom}    %domain of an operator
\DeclareMathOperator{\supp}{supp}  %support of a function
\makeatletter
\def\eqnarray{\stepcounter{equation}\let\@currentlabel=\theequation
\global\@eqnswtrue
\tabskip\@centering\let\\=\@eqncr
$$\halign to \displaywidth\bgroup\hfil\global\@eqcnt\z@
  $\displaystyle\tabskip\z@{##}$&\global\@eqcnt\@ne
  \hfil$\displaystyle{{}##{}}$\hfil
  &\global\@eqcnt\tw@ $\displaystyle{##}$\hfil
  \tabskip\@centering&\llap{##}\tabskip\z@\cr}

\def\endeqnarray{\@@eqncr\egroup
      \global\advance\c@equation\m@ne$$\global\@ignoretrue}

\def\@yeqncr{\@ifnextchar [{\@xeqncr}{\@xeqncr[5pt]}}
\makeatother

\newcommand{\ca}{{\mathcal{A}}}
\newcommand{\cd}{{\mathcal{D}}}
\newcommand{\cf}{{\mathcal{F}}}
\newcommand{\cl}{{\mathcal{L}}}

\begin{document}
\bibliographystyle{tom}

\thispagestyle{empty}

\vspace*{1cm}
\begin{center}
{\large\bf Consistent operator semigroups and their interpolation} \\[5mm]
\large  A.F.M. ter Elst and J. Rehberg

\end{center}

\vspace{5mm}

\begin{list}{}{\leftmargin=1.8cm \rightmargin=1.8cm \listparindent=10mm 
   \parsep=0pt}
\item
\small
{\sc Abstract}.
Under a mild regularity condition 
we prove that the generator of the interpolation of two $C_0$-semigroups is 
the interpolation of the two generators.
\end{list}

\let\thefootnote\relax\footnotetext{
\begin{tabular}{@{}l}
\end{tabular}}

\section{Introduction} \label{Scoin1}

Interpolation is a main tool in parabolic differential equations and in 
particular in semigroup theory, see 
\cite{BB},  \cite[Section~1.13]{Tri} and \cite[Chapter 2]{Lun}.
Frequently interpolation is done between two $L^p$-spaces or between 
a Banach space and the domain of a power of the generator of a semigroup.
The aim of this paper is to consider abstractly interpolation of continuous
semigroups, from the viewpoint of category theory.
In one of the main theorems of this paper, Theorem~\ref{tcoin104}, we
show that the generator of the interpolation of two $C_0$-semigroups is 
the interpolation of the two generators.
As a corollary this gives the following theorem for complex interpolation.

\begin{theorem} \label{tcoin130}
Let $(X,\ca,\mu)$ be a $\sigma$-finite measure space and let $p_0,p_1 \in [1,\infty)$.
Let $S^{(p_0)}$ and $S^{(p_1)}$ be bounded consistent $C_0$-semigroups in $L^{p_0}$
and $L^{p_1}$ with generators $-A_{p_0}$ and $-A_{p_1}$, respectively. 
Let $\theta \in [0,1]$ and let $p \in [1,\infty)$ be such that 
$\frac{1}{p} = \frac{1-\theta}{p_0} + \frac{\theta}{p_1}$.
Let $S^{(p)}$ be the $C_0$-semigroup on $L^p$ which is consistent with $S^{(p_0)}$.
Let $-A_p$ be the generator of $S^{(p)}$.
Then 
\[
 [D(A_{p_0}),D(A_{p_1})]_\theta = D(A_p)
{}.  \]
\end{theorem}

In order to illustrate the abstract setting of the paper we give an example
in non-linear parabolic equations, where the appropriate interpolation is 
not between two $L^p$-spaces or between a Banach space and a power of a 
semigroup generator.
Consider the quasilinear initial boundary value problem
\begin{equation} \label{e-gleichiN}
u' - \nabla \cdot \phi(u)  \nabla u  =  |\nabla u|^2, 
   \quad \nu \cdot \nabla u|_{\partial \Omega}=f \neq 0,   
   \quad  u(0)=0.
\end{equation}
on a three-dimensional (possibly nonsmooth) domain $\Omega$, where 
$u \in L^s((0,T);X)$. 
We wish to find a suitable Banach space $X$ for the treatment of this initial 
boundary value problem.
First, if the boundary conditions are replaced by a Dirichlet condition,
then $L^p$-spaces are an adequate choice for $X$ in order to treat \eqref{e-gleichiN}, 
see \cite{MERS} and \cite{HiebR}.
In view of the \emph{inhomogeneous}
Neumann conditions the $L^p$-spaces are not suitable
(cf.\ \cite[Section~1.2]{cia} \cite[Subsection~II.2.2]{GGZ}).
Secondly, at a first glance, the choice $X=W^{-1,2}(\Omega)$, which is the dual of $W^{1,2}(\Omega)$, 
seems to be adequate for the problem \eqref{e-gleichiN}, as in the linear
case, compare \cite[Section 3.3]{Lio3}.
With this choice, however, for every fixed $t$ the function $u(t,\cdot)$ is 
then in general an element of $W^{1,2}(\Omega)$ and hence 
$|\nabla u(t,\cdot)|^2$ is then an element of $L^1(\Omega)$ and fails to be an
element of $X$, as required by the differential equation in~(\ref{e-gleichiN}).
Thirdly, replacing $W^{-1,2}(\Omega)$ by the smaller space $W^{-1,q}(\Omega)$,
with $q$ larger than the spatial dimension~$3$, and under the condition 
that $-\Delta+I \colon W^{1,q}(\Omega) \to W^{-1,q}(\Omega)$ is a topological isomorphism 
(cf.\ \cite{Zanger} and \cite{HHKRZ}) one
could now guess that $X:=W^{-1,q}(\Omega)$ is a good space to treat \eqref{e-gleichiN}.
Indeed, one could then reflect the inhomogeneous Neumann condition adequately.
Moreover, since $u(t,\cdot) \in \dom(\Delta) = W^{1,q}(\Omega)$ 
one deduces that $|\nabla u(t,\cdot)|^2 \in L^{q/2}(\Omega) \subset W^{-1,q}(\Omega)$,
where we used that $q$ is larger than the space dimension~$3$.
This means that $|\nabla u(t,\cdot)|^2 \in X$ for every element $u(t,\cdot)$ in the 
domain of the elliptic
operator and each $t > 0$.
But, unfortunately, the theory for quasi-linear parabolic equations requires more: 
for elements from an interpolation space between the
Banach space and the domain of the elliptic operator the right hand side 
$|\nabla u|^2$ of the differential equation in~(\ref{e-gleichiN})
has to be 
well-behaved in order to assure at least (local in time) existence and uniqueness, 
see for example \cite[Chapter 7]{Lun}.
But for elements $v$ of an interpolation space between $W^{-1,q}(\Omega)$ and
 $W^{1,q}(\Omega)$, the gradient $\nabla v$ cannot be expected to be a function in general, and 
$|\nabla v|^2$ cannot be defined 
properly for such elements (for the required interpolation arguments see \cite{GGKR}).

In order to get out of this dilemma it turns out that, quite in coincidence with 
the concept in \cite{ClementSimonett}, that 
one should take for $X$ a Banach space which fulfils the following three properties.
\begin{tabelR}
\item \label{Scoin1-1}
$\bigl( -\Delta +I\bigl) ^{-\alpha} \colon X \to W^{1,q}$ is \emph{continuous}
for some $\alpha \in (0,1)$ and $q > 3$,
\item \label{Scoin1-2}
$|\nabla \psi|^2 \in X$ for all $\psi \in W^{1,q}$ and 
\item \label{Scoin1-3}
$-\Delta$ generates on $X$ a (suitably regular) semigroup.
\end{tabelR}

In \cite[Section 7]{HaR} a comprehensive treatment for $\eqref{e-gleichiN}$ 
is given in a (dual) Bessel function space $H^{-\tau,q}=[L^q,W^{-1,q}]_\theta$ with 
$\tau \in (\frac {3}{q},1)$.
But in the meanwhile it turned out that passing to spaces $[L^p,W^{-1,q}]_\theta$ with
$p \neq q$ gives the theory more flexibility - and sharper results. 
In particular, $p=\frac {q}{2}$ is
of special interest, see \cite{BonifaciusNeitzelRehberg}.
 The reason of this are the
better multiplier properties in the expression $\nabla \cdot \phi \nabla$ in the 
dependence of $\phi$.
Note that then the space $X=[L^p,W^{-1,q}]_\theta$ is generally \emph{not}
an interpolation space between the general Banach space $W^{-1,q}$
and the corresponding domain of the Laplacian (as a reference operator) on $W^{-1,q}$.
It turns out that indeed Properties~\ref{Scoin1-1}, \ref{Scoin1-2} and \ref{Scoin1-3} 
are satisfied when taking $X=[L^p,W^{-1,q}]_\theta$ and $p$ is suitably chosen. 
In particular, the semigroup generator property for the operators
$-\nabla \cdot \phi \nabla$ on $X=[L^p,W^{-1,q}]_\theta$
is expected to follow from the generator property on both $L^p$ and $W^{-1,q}$ by interpolation.
Indeed, this is fairly clear in case of 
\emph{analytic} semigroups, but not in general.
This is one of our motivations to investigate
interpolation properties for suitable operator semigroups in a more general context.

Since one is usually interested in a concrete equation,
 one must first know that  the two semigroups act consistently on 
the involved Banach spaces.
In Section~\ref{Scoin2} we characterise consistency of semigroups in terms 
of their resolvents and we obtain a useful expression for the intersection
of the domain of the generators.
In Section~\ref{Scoin3} we consider interpolation functors and prove 
loosely speaking that semigroup generators and interpolation functors
commute.
In the last section we give a couple of examples in $L^p$-spaces and 
distribution spaces for consistent semigroups.

\section{Consistency of operator semigroups} \label{Scoin2}

In this section we show that two $C_0$-semigroups are consistent if and only 
if the resolvents of the generators are consistent for large $\lambda > 0$.
We start we the definition of consistent operators.

\begin{definition} \label{dcoin210}
Let $X$ and $Y$ be two vector spaces.
Let $T_0 \colon D(T_0) \to Y$ and $T_1 \colon D(T_1) \to Y$ be two (linear) operators
with domains $D(T_0) \subset X$ and $D(T_1) \subset X$.
Then the operators $T_0$ and $T_1$ are called {\bf consistent} if $T_0 x = T_1 x$
for all $x \in D(T_0) \cap D(T_1)$.
Let $X_0$ and $X_1$ be two Banach spaces which are embedded in a vector space~$X$.
Let $S^{(0)}$ and $S^{(1)}$ be semigroups in $X_0$ and $X_1$, respectively. 
Then the semigroups $S^{(0)}$ and $S^{(1)}$ are called 
{\bf consistent} if $S^{(0)}_t$ and $S^{(1)}_t$ are consistent 
for all $t > 0$.
\end{definition}

The following easy lemma gives a sufficient condition for two bounded 
operators to be consistent.

\begin{lemma} \label{lcoic211}
Let $(X_0, X_1)$ be an interpolation couple of Banach spaces.
Let $T_0$ and $T_1$ be bounded operators in $X_0$ and $X_1$, respectively.
Let $D \subset X_0 \cap X_1$ and suppose that $D$ is dense in $X_0 \cap X_1$.
Further, suppose that $T_0 x = T_1 x$ for all $x \in D$.
Then $T_0$ and $T_1$ are consistent.
\end{lemma}

The boundedness condition on the semigroups in the sequel is just for convenience.

\begin{lemma} \label{lcoin101}
Let $(X_0, X_1)$ be an interpolation couple of Banach spaces.
Let $S^{(0)}$ and $S^{(1)}$ be bounded $C_0$-semigroups in $X_0$ and
$X_1$ with generators $-A_0$ and $-A_1$, respectively. 
Then the following are equivalent.
\begin{tabeleq}
\item \label{lcoin101-1}
The semigroups $S^{(0)}$ and $S^{(1)}$ are consistent.
\item \label{lcoin101-2}
For all $\lambda > 0$ the resolvent operators 
$(A_0 + \lambda \, I)^{-1}$ and $(A_1 + \lambda \, I)^{-1}$
are consistent.
\end{tabeleq}
\end{lemma}
\begin{proof}
`\ref{lcoin101-1}$\Rightarrow$\ref{lcoin101-2}'.
Let $\lambda > 0$ and $x \in X_0  \cap X_1$.
Then 
\[
(A_0 + \lambda \, I)^{-1}x 
= \int_0^\infty e^{-\lambda t} \, S^{(0)}_t x \, dt
= \int_0^\infty e^{-\lambda t} \, S^{(1)}_t x \, dt 
= (A_1 + \lambda \, I)^{-1} x.
\]

`\ref{lcoin101-2}$\Rightarrow$\ref{lcoin101-1}'.
Let $\lambda > 0$ and $x \in X_0 \cap X_1$.
Then it follows by induction to $n$ that 
$(A_0 + \lambda \, I)^{-n} x = (A_1 + \lambda \, I)^{-n} x$
for all $n \in \Ni$.
Now let $t > 0$ and $x \in X_0 \cap X_1$.
Then the Euler formula gives
\[
S^{(0)}_t x
= \lim_{n \to \infty} (A_0 + \tfrac{t}{n} \, I)^{-n} x
= \lim_{n \to \infty} (A_1 + \tfrac{t}{n} \, I)^{-n} x
= S^{(1)}_t x,
\]
as required.
\end{proof}

\begin{rem} \label{r-resolv}
In \cite{Are2} Proposition~2.2 the following is proved: 
the set $\mathcal U$ of all $\lambda$ for which
$(A_0 + \lambda \, I)^{-1}$ and $(A_1 + \lambda \, I)^{-1}$ are consistent, is open and closed in
$\rho(-A_0) \cap \rho(-A_1)$. 
From this it easily follows that if $\rho(-A_0)=\rho(-A_1)$ and this set 
is connected, then the consistency of $(A_0 + \lambda_0 \, I)^{-1}$ and 
$(A_1 + \lambda_0 \, I)^{-1}$
for \emph{only one} $\lambda_0$ implies the consistency of all resolvent operators.
\end{rem}

If the equivalent conditions in Lemma~\ref{lcoin101} are valid, then 
it is possible that there exists a $\lambda \in \rho(-A_0) \cap \rho(-A_1)$
such that the resolvents 
$(A_0 + \lambda \, I)^{-1}$ and $(A_1 + \lambda \, I)^{-1}$
are not consistent.
An example has been given in \cite{Are2} Section~3.

\begin{proposition} \label{pcoin102}
Let $(X_0, X_1)$ be an interpolation couple of Banach spaces.
Let $S^{(0)}$ and $S^{(1)}$ be bounded consistent $C_0$-semigroups in $X_0$ and
$X_1$ with generators $-A_0$ and $-A_1$, respectively. 
Then one has the following.
\begin{tabel}
\item \label{pcoin102-1}
The generators $A_0$ and $A_1$ are consistent.
\item \label{pcoin102-2}
$D(A_0) \cap D(A_1) = \{ x \in D(A_0) \cap X_1 : A_0 x \in X_1 \}
= (A_0 + I)^{-1}(X_0 \cap X_1)$.
\end{tabel}
\end{proposition}
\begin{proof}
`\ref{pcoin102-1}'.
Let $x \in D(A_0) \cap D(A_1)$ and $F \in (X_0 + X_1)'$.
Then 
\[
F(A_0 x)
= \lim_{t \downarrow 0} \tfrac{1}{t} F((I - S^{(0)})x)
= \lim_{t \downarrow 0} \tfrac{1}{t} F((I - S^{(1)})x)
= F(A_1 x)
 .  \]
Hence $A_0 x = A_1 x$.

`\ref{pcoin102-2}'.
Let $x \in D(A_0) \cap D(A_1)$.
Then it follows from Statement~\ref{pcoin102-1} that 
$A_0 x = A_1 x \in X_1$.
So $D(A_0) \cap D(A_1)
\subset \{ x \in D(A_0) \cap X_1 : A_0 x \in X_1 \} $.
Conversely, suppose $x \in D(A_0) \cap X_1$ and $A_0 x \in X_1$.
Let $t > 0$.
Then for all $F \in (X_0 + X_1)'$ one deduces that
\begin{eqnarray*}
F\Big( (I - S^{(1)}_t) x \Big)
& = & F\Big( (I - S^{(0)}_t) x \Big)
= F\Big( \int_0^t S^{(0)}_s \, A_0 x \, ds \Big)
= \int_0^t F\Big( S^{(0)}_s \, A_0 x \Big) \, ds  \\
& = & \int_0^t F\Big( S^{(1)}_s \, A_0 x \Big) \, ds
= F\Big( \int_0^t S^{(1)}_s \, A_0 x \, ds \Big) .
\end{eqnarray*}
So 
\[
\tfrac{1}{t} (I - S^{(1)}_t) x
= \tfrac{1}{t} \int_0^t S^{(1)}_s \, A_0 x \, ds
\]
in $X_1$.
Hence 
\[
\lim_{t \downarrow 0} \tfrac{1}{t} (I - S^{(1)}_t) x
= A_0 x
\]
in $X_1$.
Therefore $x \in D(A_1)$.
This proves the first equality in Statement~\ref{pcoin102-2}.

Next, let $x \in D(A_0) \cap D(A_1)$.
Then $(A_0 + I) x = (A_1 + I) x \in X_0 \cap X_1$
again by Statement~\ref{pcoin102-1}.
So $x \in (A_0 + I)^{-1}(X_0 \cap X_1)$.
Conversely, let $x \in X_0 \cap X_1$.
Then obviously $(A_0 + I)^{-1} u \in D(A_0)$.
Since $(A_0 + I)^{-1}$ and $(A_1 + I)^{-1}$ are consistent by 
Lemma~\ref{lcoin101}, it follows that 
$(A_0 + I)^{-1} x = (A_1 + I)^{-1} x \in D(A_1)$.
So $(A_0 + I)^{-1} x \in D(A_0) \cap D(A_1)$.
\end{proof}

\section{Interpolation of consistent operator semigroups} \label{Scoin3}

In this section we consider interpolation of semigroups and their generators.
In all what follows, we adopt the terminology of \cite{Tri} Section~1.2, 
with minor modifications. 

Let $(X_0,X_1)$ and $(Y_0,Y_1)$ be two interpolation couples of Banach spaces.
Recall from \cite{Tri} Subsection~1.2.2 that 
$L( (X_0,X_1) , (Y_0,Y_1) )$ denotes the vector space of all 
linear maps $T \colon X_0 + X_1 \to Y_0 + Y_1$ such that 
$T|_{X_0} \in \cl(X_0,Y_0)$ and $T|_{X_1} \in \cl(X_1,Y_1)$.
Clearly the operators $T|_{X_0}$ and $T|_{X_1}$ are consistent for all 
$T \in L( (X_0,X_1) , (Y_0,Y_1) )$.
There is a converse.

\begin{lemma} \label{l-induce}
Let $(X_0,X_1)$ and $(Y_0,Y_1)$ be two interpolation couples of Banach spaces,
$T_0 \in \cl(X_0,Y_0)$ and $T_1 \in \cl(X_1,Y_1)$.
Suppose that $T_0$ and $T_1$ are consistent.
Then there exists a unique $T \in L( (X_0,X_1) , (Y_0,Y_1) )$
such that $T|_{X_0} = T_0$ and $T|_{X_1} = T_1$.

Moreover, the operator $T$ is continuous from $X_0 + X_1$ into $Y_0 + Y_1$
and $\|T\|_{X_0 + X_1 \to Y_0 + Y_1} \leq \|T_0\|_{X_0 \to Y_0} \vee \|T_1\|_{X_1 \to Y_1}$.
\end{lemma}
\begin{proof}
The first part is easy and $T \in L( (X_0,X_1) , (Y_0,Y_1) )$ is given 
by $T(x_0 + x_1) = T_0 x_0 + T_1 x_1$ for all $x_0 \in X_0$ and $x_1 \in X_1$.
Here we use that $T_0$ and $T_1$ are consistent.

Next, let $x \in X_0 + X_1$.
Let $x_0 \in X_0$ and $x_1 \in X_1$ be such that $x = x_0 + x_1$.
Then
\begin{eqnarray*}
\|T x\|_{X_0 + X_1}
& \leq & \|T_0 x_0\|_{X_0} + \|T_1 x_1\|_{X_1}  \\
& \leq & ( \|T_0\|_{X_0 \to Y_0} \vee \|T_1\|_{X_1 \to Y_1}) (\|x_0\|_{X_0} + \|x_1\|_{X_1})  
 .
\end{eqnarray*}
So $\|T x\|_{Y_0 + Y_1} 
\leq ( \|T_0\|_{X_0 \to Y_0} \vee \|T_1\|_{X_1 \to Y_1}) \|x\|_{X_0 + X_1} $.
This proves the last assertion.
\end{proof}

Let $(X_0,X_1)$ and $(Y_0,Y_1)$ be two interpolation couples of Banach spaces.
We provide $L( (X_0,X_1) , (Y_0,Y_1) )$ with the norm
\[
\|T\|_{L( (X_0,X_1) , (Y_0,Y_1) )} 
= \|T|_{X_0}\|_{X_0 \to Y_0} \vee \|T|_{X_1}\|_{X_1 \to Y_1}
 .  \]
Then $L( (X_0,X_1) , (Y_0,Y_1) )$ is a Banach space.
For the concept of interpolation functor we refer to \cite{Tri} Subsection~1.2.2.
If $\cf$ is an interpolation functor and $T \in L( (X_0,X_1) , (Y_0,Y_1) )$,
then we denote by $T^\cf \colon \cf(X_0,X_1) \to \cf(Y_0,Y_1)$ the restriction 
of $T$ to $\cf(X_0,X_1)$. 
Note that $T^\cf$ is a bounded operator. 
Alternatively, since we are interested in consistent operators, we also introduce
another notation.
Let $T_0 \in \cl(X_0,Y_0)$ and $T_1 \in \cl(X_1,Y_1)$.
Suppose that $T_0$ and $T_1$ are consistent.
By Lemma~\ref{l-induce} there exists a unique $T \in L( (X_0,X_1) , (Y_0,Y_1) )$
such that $T|_{X_0} = T_0$ and $T|_{X_1} = T_1$.
Then we define 
\[
\cf(T_0,T_1) = T^\cf
 .  \]
So $\cf(T_0,T_1)$ is a bounded operator from $\cf(X_0,X_1)$ into $\cf(Y_0,Y_1)$.
Since $T_0$, $T_1$ and $\cf(T_0,T_1) = T^\cf$ are all three restrictions of the 
same operator $T$ on $X_0 + X_1$, it is obvious that the three operators
$T_0$, $T_1$ and $\cf(T_0,T_1) = T^\cf$ are pairwise consistent.

\begin{lemma} \label{lcoin303}
Let $(X_0,X_1)$ and $(Y_0,Y_1)$ be two interpolation couples of Banach spaces and 
$\cf$ an interpolation functor.
Then there exists an $M > 0$ such that 
\[
\|T^\cf\|_{ \cf(X_0,X_1) \to \cf(Y_0,Y_1) }
\leq M \, \|T\|_{ L( (X_0,X_1) , (Y_0,Y_1) ) }
\]
for all $T \in L( (X_0,X_1) , (Y_0,Y_1) )$.
\end{lemma}
\begin{proof}
The operator $T \mapsto T^\cf$ from the Banach space $L( (X_0,X_1) , (Y_0,Y_1) )$ into 
the Banach space $\cl( \cf(X_0,X_1) , \cf(Y_0,Y_1) )$ has a closed graph.
\end{proof}

In several contexts \emph{dual} semigroups are of interest, see the papers
\cite{Ama}, \cite{AmE}.
Therefore it makes sense to establish a connection between consistency of operators 
and consistency of their adjoints.

\begin{proposition} \label{p-adjoint}
Let $(X_0,X_1)$ and $(Y_0,Y_1)$ be two interpolation couples of Banach spaces,
$T_0 \in \cl(X_0,Y_0)$ and $T_1 \in \cl(X_1,Y_1)$.
Suppose that $T_0$ and $T_1$ are consistent.
Let $T \in L((X_0,X_1),(Y_0,Y_1))$ be such that $T|_{X_0} = T_0$ and 
$T|_{X_1} = T_1$.
Then one has the following.
\begin{tabel}
\item \label{pcoin66-1}
$T'=T_0'|_{(Y_0 +Y_1)'} = T_1'|_{(Y_0 +Y_1)'}$.

\item \label{pcoin67-1}
If $Y_0 \cap Y_1$ is dense in both spaces $Y_0$ and $Y_1$, then
\begin{equation} \label{e-coooi}
(Y_0 + Y_1)' = Y_0' \cap Y_1'
\end{equation}
and, consequently,
\begin{equation} \label{e-restri}
T'=T_0'|_{ Y_0' \cap Y_1'} = T_1'|_{ Y_0' \cap Y_1'}.
\end{equation}
\end{tabel}
\end{proposition}
\begin{proof}
`\ref{pcoin66-1}'.
 Clearly the adjoint $T'$ of $T$ is a continuous operator from 
$\bigl( Y_0 +Y_1)' $ into $\bigl( X_0 +X_1)'$. 
Let $f  \in \bigl( Y_0+Y_1\bigl) ' \subset Y_0'$ and $x \in X_0 \subset X_0 +X_1$.
Then
\[
\langle T_0' f,x \rangle_{X_0' \times X_0} 
=\langle  f,T_0x \rangle_{Y_0' \times Y_0} 
= \langle  f,Tx \rangle_{(Y_0+Y_1)' \times (Y_0+Y_1)}
=\langle  T'f,x \rangle_{(X_0+X_1)' \times (X_0+X_1)}.
\]
The second equality is proved analogously.

`\ref{pcoin67-1}'.
 Under the density condition, the equality \eqref{e-coooi} is well-known,
cf.\ \cite{BL} Theorem~2.7.1.
Then (\ref{e-restri}) follows from \ref{pcoin66-1}.
\end{proof}

\begin{definition} \label{d-dtype}
We say that an interpolation functor $\mathcal F$ has {\bf Property~(d)} (for dense)
if for every interpolation couple $(X_0,X_1)$ the subspace
$X_0 \cap X_1$ is \emph{dense} in the interpolation space $\mathcal F(X_0,X_1)$.
\end{definition}

\begin{exam} \label{xcoin301}
The complex interpolation has Property~(d).
With exception of the limit values also the real interpolation has Property~(d).
For complex and real interpolation, see \cite{Tri} Subsections~1.9.3 and~1.6.2.
\end{exam}

\begin{exam} \label{xcoin302}
The real interpolation with parameters the limit values does not have
Property~(d), see \cite{Tri} Remark~1.18.3.5.
\end{exam}

The next lemma is easy to prove.

\begin{lemma} \label{lcoin320}
Let $(X_0,X_1)$ and $(Y_0,Y_1)$ be two interpolation couples of Banach spaces and 
$\cf$ an interpolation functor which has Property~(d).
Let $T_0 \in \cl(X_0,Y_0)$,  $T_1 \in \cl(X_1,Y_1)$ and suppose 
that $T_0$ and $T_1$ are consistent.
Then $\cf(T_0,T_1)$ is the {\em unique} 
extension of the operator $T|_{X_0 \cap X_1} \colon X_0 \cap X_1 \to Y_0 \cap Y_1$ 
which is continuous from the space $\cf(X_0,X_1)$ into the space $\cf(Y_0,Y_1)$.
\end{lemma}

Next we consider a functor on consistent semigroups.

\begin{proposition} \label{pcoin1030}
Let $\mathcal F$ be an interpolation functor.
Let $(X_0,X_1)$ be an interpolation couple of Banach spaces.
Let $S^{(0)}$ and $S^{(1)}$ be consistent
semigroups in $X_0$ and $X_1$ respectively. 
Then one has the following.
\begin{tabel}
\item \label{pcoin1030-1}
The family $\Big( \cf( S^{(0)}_t, S^{(1)}_t) \Big)_{t > 0}$
on $\cf(X_0,X_1)$ is a semigroup which is consistent 
with both $S^{(0)}$ and $S^{(1)}$.
\item \label{pcoin1030-1.5}
If both $S^{(0)}$ and $S^{(1)}$ are bounded semigroups, then 
the semigroup $\Big( \cf( S^{(0)}_t, S^{(1)}_t) \Big)_{t > 0}$ is also bounded.
\item \label{pcoin1030-2}
Suppose in addition that $S^{(0)}$ and $S^{(1)}$ are $C_0$-semigroups and that 
the interpolation functor $\mathcal F$ has Property~(d).
Then the semigroup $\Big( \cf( S^{(0)}_t, S^{(1)}_t) \Big)_{t > 0}$ is a $C_0$-semigroup.
\end{tabel}
\end{proposition}
\begin{proof}
`\ref{pcoin1030-1}'.
This is straightforward.

`\ref{pcoin1030-1.5}'.
This follows from Lemmas~\ref{l-induce} and \ref{lcoin303}.

`\ref{pcoin1030-2}'.
Without loss of generality we may assume that both $S^{(0)}$ and
$S^{(1)}$ are bounded semigroups.
For all $t > 0$ write $S^\cf_t = \cf( S^{(0)}_t, S^{(1)}_t)$.
Then also $(S^\cf_t)_{t > 0}$ is a bounded semigroup by Statement~\ref{pcoin1030-1.5}.

Since $\cf(X_0,X_1)$ is an intermediate space for the interpolation couple
$(X_0,X_1)$, there exists a $c > 0$ such that 
$\|x\|_{\cf(X_0,X_1)} \leq c \, \|x\|_{X_0 \cap X_1}$
for all $x \in X_0 \cap X_1$.
Let $x \in X_0 \cap X_1$ and $t > 0$.
Then 
\[
\|S^\cf_t x - x\|_{\cf(X_0,X_1)}
\leq c \, \|S^\cf_t x - x\|_{X_0 \cap X_1}
= c \, (\|S^{(0)}_t x - x\|_{X_0} + \|S^{(1)}_t x - x\|_{X_1})
 .  \]
Hence $\lim_{t \downarrow 0} \|S^\cf_t x - x\|_{\cf(X_0,X_1)} = 0$ and 
$\lim_{t \downarrow 0} S^\cf_t x = x$ in $\cf(X_0,X_1)$.

Finally, $X_0 \cap X_1$ is dense in $\cf(X_0,X_1)$ since the interpolation
functor has Property~(d).
So $\lim_{t \downarrow 0} S^\cf_t x = x$ in $\cf(X_0,X_1)$ for all $x \in \cf(X_0,X_1)$.
\end{proof}

We wish to determine the generator of the semigroup $S^\cf$.
We need a lemma.

\begin{lemma} \label{lcoin103}
Let $\mathcal F$ be an interpolation functor which has Property~(d).
Let $(X_0,X_1)$ be an interpolation couple of Banach spaces.
Further, let $S^{(0)}$ and $S^{(1)}$ be consistent
$C_0$-semigroups in $X_0$ and $X_1$ with generators $-A_0$ and $-A_1$, 
respectively. 
Let $S^\cf = \Big( \cf( S^{(0)}_t, S^{(1)}_t) \Big)_{t > 0}$ 
be the $C_0$-semigroup in $\cf(X_0,X_1)$ as in 
Proposition~\ref{pcoin1030}.
Let $-B$ be the generator of $S^{\mathcal F}$.
Then $D(A_0) \cap D(A_1) \subset D(B)$ and 
$D(A_0) \cap D(A_1)$ is a core for $B$.
\end{lemma}
\begin{proof}
Without loss of generality we may assume that both $S^{(0)}$ and 
$S^{(1)}$ are bounded semigroups.
The resolvent
\[
(B+I)^{-1} \colon \mathcal F(X_0,X_1) \to D(B)
\]
is a topological isomorphism.
Also the resolvent operators $(B + I)^{-1}$ and $(A_0 + I)^{-1}$
are consistent by Proposition~\ref{pcoin1030}\ref{pcoin1030-1} and Lemma~\ref{lcoin101}.
By Lemma~\ref{pcoin102}\ref{pcoin102-2} the restriction
\[
(B+I)^{-1}|_{X_0 \cap X_1} 
= (A_0 + I)^{-1}|_{X_0 \cap X_1} \colon X_0 \cap X_1 \to D(A_0) \cap D(A_1)
\]
is a bijection.
Because $X_0 \cap X_1 \subset \mathcal F(X_0,X_1)$, this 
implies immediately the assertion  $D(A_0) \cap D(A_1) \subset D(B)$.
Since $\mathcal F$ has Property~(d),
the space $X_0 \cap X_1$ is dense in $\mathcal F(X_0,X_1)$.
Hence $ D(A_0) \cap D(A_1)$ is dense in $ D(B)$.
\end{proof}

We provide the domain of a generator with the graph norm.
Note that with the notation of the previous lemma, 
$(D(A_0),D(A_1))$ is an interpolation couple and 
$A_0 \in \cl(D(A_0),X_0)$ and similarly $A_1 \in \cl(D(A_1),X_1)$.
Now we are able to prove the main theorem of this paper.

\begin{theorem} \label{tcoin104}
Let $\mathcal F$ be an interpolation functor which has Property~(d).
Let $(X_0,X_1)$ be an interpolation couple of Banach spaces.
Further, let $S^{(0)}$ and $S^{(1)}$ be consistent
$C_0$-semigroups in $X_0$ and $X_1$ with generators $-A_0$ and $-A_1$, 
respectively. 
Then $- \cf(A_0,A_1)$ is the generator of the semigroup
$\Big( \cf( S^{(0)}_t, S^{(1)}_t) \Big)_{t > 0}$.

In particular, 
\[
D(\cf(A_0,A_1)) = \mathcal F(D(A_0),D(A_1)).
\]
\end{theorem} 
\begin{proof}
Without loss of generality we may assume that $S^{(0)}$ and $S^{(1)}$
are bounded semigroups.
Write $S^\cf_t = \cf( S^{(0)}_t, S^{(1)}_t)$ for all $t > 0$ and let 
$-B$ be the generator of the $C_0$-semigroup $S^\cf$.
We know that $D(A_0) \cap D(A_1) \subset D(B)$ by Lemma~\ref{lcoin103}.
Also $B x = A_0 x = A^\cf x$ for all $x \in D(A_0) \cap D(A_1)$
by Proposition~\ref{pcoin102}\ref{pcoin102-1}, where we set 
$A^\cf = \cf(A_0,A_1)$.
The operator $A^\cf$ is bounded from 
$\cf(D(A_0),D(A_1))$ into $\cf(X_0,X_1)$.
Hence there exists a $c > 0$ such that 
\[
\|A^\cf x\|_{\cf(X_0,X_1)} 
\leq c \, \|x\|_{\cf(D(A_0),D(A_1))}
\]
for all $x \in \cf(D(A_0),D(A_1))$.
If $x \in D(A_0) \cap D(A_1)$, then $B x = A^\cf x$ and
\[
\|B x\|_{\cf(X_0,X_1)} \leq c \, \|x\|_{\cf(D(A_0),D(A_1))}
 .  \]
Let $x \in \cf(D(A_0),D(A_1))$.
Since $D(A_0) \cap D(A_1)$ is dense in $\cf(D(A_0),D(A_1))$ by Property~(d),
there exists a sequence $(x_n)_{n \in \Ni}$ in $D(A_0) \cap D(A_1)$
such that $\lim x_n = x$ in $\cf(D(A_0),D(A_1))$.
Then $(B x_n)_{n \in \Ni}$ is a Cauchy sequence in $\cf(X_0,X_1)$
and $\lim x_n = x$ in $\cf(X_0,X_1)$.
Since $B$ is a closed operator, it follows that $x \in D(B)$ and 
$B x = \lim B x_n = \lim A^\cf x_n = A^\cf x$ in $\cf(X_0,X_1)$.
Hence $B$ is an extension of $A^\cf$.

It remains to show that $D(B) \subset \cf(D(A_0),D(A_1))$.
The operator $(A_0 + I)^{-1}$ is bounded from $X_0$ into $D(A_0)$
and the operator $(A_1 + I)^{-1}$ is bounded from $X_1$ into $D(A_1)$.
Moreover, the operators $(A_0 + I)^{-1}$ and $(A_1 + I)^{-1}$
are consistent by Lemma~\ref{lcoin101}.
So by interpolation one obtains a bounded operator, denoted by $C$, 
from $\cf(X_0,X_1)$ into 
$\cf(D(A_0),D(A_1))$.
Let $c' > 0$ be such that 
\[
\|C x\|_{\cf(D(A_0),D(A_1))} 
\leq c' \, \|x\|_{\cf(X_0,X_1)}
\]
for all $x \in \cf(X_0,X_1)$.
If $x \in X_0 \cap X_1$, then 
$C x = (A_0 + I)^{-1} x$.
Hence 
\[
\|(A_0 + I)^{-1} x\|_{\cf(D(A_0),D(A_1))}
\leq c' \, \|x\|_{\cf(X_0,X_1)}
\]
for all $x \in X_0 \cap X_1$.
Using Proposition~\ref{pcoin102}\ref{pcoin102-2} it follows that 
\[
\|x\|_{\cf(D(A_0),D(A_1))} 
\leq c' \, \|(A_0 + I) x\|_{\cf(X_0,X_1)}
= c' \, \|(B + I) x\|_{\cf(X_0,X_1)}
\]
for all $x \in D(A_0) \cap D(A_1)$.
But $D(A_0) \cap D(A_1)$ is dense in $D(B)$ by Lemma~\ref{lcoin103}.
Since $\cf(D(A_0),D(A_1))$ is complete, it follows that 
$D(B) \subset \cf(D(A_0),D(A_1))$.
\end{proof}
A similar statement is valid for the resolvents.

\begin{proposition} \label{pcoin105}
Let $\mathcal F$ be an interpolation functor which has Property~(d).
Let $(X_0,X_1)$ be an interpolation couple of Banach spaces.
Further, let $S^{(0)}$ and $S^{(1)}$ be consistent bounded
$C_0$-semigroups in $X_0$ and $X_1$ with generators $-A_0$ and $-A_1$, 
respectively. 
Then 
\[
\cf\Big( (A_0 + I)^{-1}, (A_1 + I)^{-1} \Big)
= \Big( \cf(A_0,A_1) + I \Big)^{-1}
 .  \]
\end{proposition} 
\begin{proof}
Write $S^\cf_t = \cf( S^{(0)}_t, S^{(1)}_t)$ for all $t > 0$.
Let $x \in X_0 \cap X_1$.
If $F \in (X_0 + X_1)'$, then 
\begin{eqnarray*}
F \Big( (\cf(A_0,A_1) + I)^{-1} x \Big)
& = & \int_0^\infty e^{-t} \, F \Big( S^\cf_t x \Big) \, dt  
= \int_0^\infty e^{-t} \, F \Big( S^{(0)}_t x \Big) \, dt  \\
& = & F \Big( (A_0 + I)^{-1} x \Big)
= F \Big( \cf\Big( (A_0 + I)^{-1}, (A_1 + I)^{-1} \Big) x \Big)  
 .  
\end{eqnarray*}
So 
\[
(\cf(A_0,A_1) + I)^{-1} x = \cf\Big( (A_0 + I)^{-1}, (A_1 + I)^{-1} \Big) x
 .  \]
Moreover, the operator $(\cf(A_0,A_1) + I)^{-1}$ is bounded from $\cf(X_0,X_1)$ into itself.
Hence $\cf\Big( (A_0 + I)^{-1}, (A_1 + I)^{-1} \Big) = (\cf(A_0,A_1) + I)^{-1}$ by Lemma~\ref{lcoin320}.
\end{proof}

\section{Example, $L^p$-spaces} \label{Scoin4}

One of the commonly used theorems states that semigroups on $L^p$-spaces, which are 
induced by forms on $L^2$, extrapolate \emph{consistently} to the whole $L^p$-scale,
provided one knows Gaussian estimates for the $L^2$-semigroup.
We next describe this situation.

Let $\Omega \subset \R^d$ be a bounded domain and $\mathcal D \subset \partial \Omega$ be closed.
We 
define
\[
C^\infty_{\mathcal D}(\Omega)
=\{\psi|_\Omega: \psi \in C^\infty(\R^d) \mbox{ and } \supp \psi \cap \mathcal D =\emptyset \}.
\]
For all $p \in [1,\infty)$ let 
$W^{1,p}_{\mathcal D}(\Omega)$ be the closure of $C^\infty_{\mathcal D}(\Omega)$ in 
$W^{1,p}(\Omega)$.
If $q \in (1,\infty]$, then we denote by $W^{-1,q}_{\mathcal D}(\Omega)$
the (anti-)dual of the space $W^{1,q'}_{\mathcal D}(\Omega)$, where $q'$ is the 
dual exponent of $q$.
Let $\mu$ be a real, bounded, measurable, elliptic function on $\Omega$ which takes 
its values in the set of real $d \times d$-matrices.
Define the sesquilinear form 
$\mathfrak t \colon W^{1,2}_{\mathcal D}(\Omega) \times W^{1,2}_{\mathcal D}(\Omega) \to \Ci$ by
\[
\mathfrak t[u,v] = \int _\Omega \mu \nabla u \cdot \overline{\nabla v}
 .  \]
Let $A$ be the operator associated with $\mathfrak t$ in $L^2(\Omega)$ and 
let $\ca \colon W^{1,2}_{\mathcal D}(\Omega) \to W^{-1,2}_{\mathcal D}(\Omega)$ be 
defined by $\langle \ca u,v \rangle = \mathfrak t[u,v]$ for all $u,v \in W^{1,2}_{\mathcal D}(\Omega)$.
Then $A$ and $\ca$ generate analytic semigroups
$S^{(2)}$ and $\widetilde S^{(2)}$ on $L^2(\Omega)$ and 
$W^{-1,2}_{\mathcal D}(\Omega)$, respectively.

\begin{theorem} \label{t-consistlp}
Adopt the above notation and assumptions.
\begin{tabel}
\item \label{t-consistlp-1}
The  semigroups $S^{(2)}$ and $\widetilde S^{(2)}$  are consistent.
\item \label{t-consistlp-2}
If the boundary around any point $x \in \overline {\partial \Omega \setminus \mathcal D}$ 
admits a 
bi-Lipschitzian boundary chart, then the semigroup $S^{(2)}$ on $L^2(\Omega)$ has a kernel 
with Gaussian upper estimates. 

Moreover, the semigroup $S^{(2)}$ extends consistently to a $C_0$-semigroup $S^{(p)}$ on 
$L^p(\Omega)$ for all $p \in [1,\infty)$. 
\end{tabel}
\end{theorem}
\begin{proof}
`\ref{t-consistlp-1}'.
See \cite{Ouh5} Subsection~1.4.2.

`\ref{t-consistlp-2}'. 
The first assertion is proved in \cite{ERe1} Theorem~3.1.
The second one follows from the
first by \cite{Are2} second proof on page~1160.
\end{proof}

It is desirable in various contexts to know the consistency of semigroups on spaces
like $L^p(\Omega)$ and $W^{-1,q}_\mathcal D(\Omega)$ -- as outlined in the introduction.
Before we prove such a result we establish the following lemma.

\begin{lemma} \label{l-embeddens}
Let $p \in [1,\infty)$ and $q \in (1,\infty)$.
Then $C_c^\infty(\Omega)$ is dense in $W^{-1,q}_{\mathcal D}(\Omega) \cap L^p(\Omega)$.
\end{lemma}
\begin{proof}
First of all, $W^{-1,q}_{\mathcal D}(\Omega) \cap L^p(\Omega)$ is dense in both  
$W^{-1,q}_{\mathcal D}(\Omega)$ and $L^p(\Omega)$, since 
$C_c^\infty(\Omega) \subset W^{-1,q}_{\mathcal D}(\Omega) \cap L^p(\Omega)$.
Therefore
\[
\bigl( W^{-1,q}_{\mathcal D}(\Omega)\cap L^p(\Omega)\bigl) ' 
= \bigl( W^{-1,q}_{\mathcal D}(\Omega)\bigl)  ' 
   +\bigl( L^p(\Omega)\bigl) '
=W^{1,q'}_{\mathcal D}(\Omega) + L^{p'}(\Omega),
\]
cf.\ \cite{BL} Theorem~2.7.1.
Let $F \in (W^{-1,q}_{\mathcal D}(\Omega)\cap L^p(\Omega))'
= W^{1,q'}_{\mathcal D}(\Omega)+ L^{p'}(\Omega)$ 
and suppose that $F(u) = 0$ for all $u \in C_c^\infty(\Omega)$.
Then $F \in L^1(\Omega)$.
So $F = 0$.
Then the statement follows from the Hahn--Banach theorem.
\end{proof}

For all $p \in [1,\infty)$ let $S^{(p)}$ be the semigroup on $L^p(\Omega)$ 
as in Theorem~\ref{t-consistlp}\ref{t-consistlp-2} (assuming the conditions 
of that theorem are satisfied).

\begin{theorem}\label{t-coincid}
Assume that the boundary around any point 
$x \in \overline {\partial \Omega \setminus \mathcal D}$ 
admits a bi-Lipschitzian boundary chart.
Then
the semigroup $\widetilde S^{(2)}$
is consistent with the semigroup $S^{(p)}$ for every $ p \in [1,\infty)$. 
\end{theorem}
\begin{proof}
Let $t > 0$.
If $u \in C_c^\infty(\Omega)$, then 
$\widetilde S^{(2)}_t u = S^{(2)}_t u = S^{(p)}_t u$ by Theorem~\ref{t-consistlp}.
Now the result follows from Lemmas~\ref{lcoic211} and \ref{l-embeddens}.
\end{proof}

There is also a version for $W^{-1,q}_{\mathcal D}(\Omega)$
with $q \in [2,\infty)$ under slightly more assumptions.

\begin{theorem}\label{tcoin405}
Suppose that 
\begin{itemize}
\item
the boundary around any point 
$x \in \overline {\partial \Omega \setminus \mathcal D}$ 
admits a bi-Lipschitzian boundary chart,
\item
the set $\cd$ is a $(d-1)$-set in the sense of Jonsson--Wallin \cite{JW} Chapter~II
and
\item
$\Omega$ is a $d$-set in the sense of Jonsson--Wallin, 
or for all $x \in \Omega$ the matrix $\mu(x)$ is symmetric.
\end{itemize}
Let $q \in [2,\infty)$.
Define the operator $\widetilde \ca_q$ in $W^{-1,q}_{\mathcal D}(\Omega)$ by
\[
D(\widetilde \ca_q)
= \{ \psi \in W^{-1,q}_{\mathcal D}(\Omega) \cap W^{1,2}_{\mathcal D}(\Omega) 
      : \ca \psi \in W^{-1,q}_{\mathcal D}(\Omega) \}
\] 
and $\widetilde \ca_q = \ca|_{D(\widetilde \ca_q)}$.
Then $- \widetilde \ca_q$ generates a holomorphic semigroup on $W^{-1,q}_{\mathcal D}(\Omega)$
which is consistent with the semigroup $S^{(p)}$ for all $p \in [1,\infty)$.
\end{theorem}
\begin{proof}
It follows from \cite{DisserterElstRehberg2} Lemma~6.9(c) that 
$- \widetilde \ca_q$ generates a holomorphic semigroup on $W^{-1,q}_{\mathcal D}(\Omega)$.
Denote this semigroup by $\widetilde S^{(q)}$. 
Then $\widetilde S^{(q)}$ is consistent with $\widetilde S^{(2)}$ by the 
paragraph before Lemma~6.9 in \cite{DisserterElstRehberg2}.
Hence if $t > 0$ and $u \in C_c^\infty(\Omega)$, then 
$\widetilde S^{(q)}_t u = \widetilde S^{(2)}_t u = S^{(p)}_t u$.
Finally use again Lemmas~\ref{lcoic211} and \ref{l-embeddens}.
\end{proof}

\subsection*{Acknowledgements}
Part of this work is supported by the Marsden Fund Council from Government funding,
administered by the Royal Society of New Zealand.

\small 

\noindent
{\sc A.F.M. ter Elst,
Department of Mathematics,
University of Auckland,
Private bag 92019,
Auckland 1142,
New Zealand}  \\
{\em E-mail address}\/: {\bf terelst@math.auckland.ac.nz}

\mbox{}

\noindent
{\sc J. Rehberg,
Weierstrass Institute for Applied Analysis and Stochastics,
Mohrenstr.~39, 
10117 Berlin, 
Germany}  \\
{\em E-mail address}\/: {\bf rehberg@wias-berlin.de}

\end{document}